\documentclass[12pt]{amsart}
\usepackage{amssymb}
\def\build#1_#2^#3{\mathrel{
\mathop{\kern 0pt#1}\limits_{#2}^{#3}}}

\def\R{\mathop{\mathbb R\kern 0pt}\nolimits}
\def\H{\mathop{\mathbb H\kern 0pt}\nolimits}
\def\C{\mathop{\mathbb C\kern 0pt}\nolimits}
\def\N{\mathop{\mathbb N\kern 0pt}\nolimits}
\def\ZZZ{\mathop{\mathbb Z\kern 0pt}\nolimits}
\def\Q{\mathop{\mathbb Q\kern 0pt}\nolimits}

\newtheorem{thm}{Theorem}[section]

\newtheorem{pro}[thm]{Proposition}
\newtheorem{lem}[thm]{Lemma}
\newtheorem{lemma}[thm]{Lemma}
\newtheorem{cor}[thm]{Corollary}
\newtheorem{remark}[thm]{Remark}

\newtheorem{defi}[thm]{Definition}

\def\R{{\mathbb R}}
\def\N{{\mathbb N}}
\def\Z{{\mathbb Z}}
\def\C{{\mathbb C}}

\def\al{\alpha}

\def\D{\Delta}

\def\X{{\mathcal  X}}

\def\la{\lambda}
\def\rh{\rho}

\def\H{{\mathcal H}}

\newtheorem{theo}{Theorem}

\def \lam{\lambda}
\def\D{\Delta}

\def \al{\alpha}

\def\virgp{\raise 2pt\hbox{,}}

\def\and{\quad\hbox{and}\quad}
\def\G{{\mathbb G}}

\begin{document}

\title[Besov algebras on Lie groups of polynomial growth]{Besov algebras on Lie groups of polynomial growth}

\author{Isabelle Gallagher}
\address{Institut de Math{\'e}matiques UMR 7586, 
      Universit{\'e} Paris VII, 175, rue du Chevaleret, 75013 Paris, France}
\email{gallagher@math.jussieu.fr}

\author{Yannick Sire}
\address{Universit\'e Paul C\'ezanne, LATP,
Facult\'e des Sciences et Techniques, Case cour A,
Avenue Escadrille Normandie-Niemen, F-13397 Marseille Cedex 20,
France, and CNRS, LATP, CMI, 39 rue F. Joliot-Curie, F-13453 Marseille Cedex
13, France}
\email{sire@cmi.univ-mrs.fr}
\thanks{The first  author is partially supported by
the ANR project ANR-08-BLAN-0301-01 "Mathoc\'ean", as well as by the Institut Universitaire de France. The second author is supported by
the ANR project "PREFERED".}

\begin{abstract}
We prove an algebra property under pointwise multiplication for   Besov spaces defined on Lie groups of polynomial growth.  When the setting is restricted to the case of~H-type groups, this algebra property is generalized to paraproduct estimates.
 \end{abstract}

\maketitle
\section{Introduction}

\subsection{Lie groups  of polynomial growth}  In this paper $\mathbb G$ is an  unimodular connected Lie group endowed with the Haar measure. By ``unimodular'' we mean that the Haar measure is left and right-invariant.  Denoting by $\mathcal  G$  the Lie algebra  of $\mathbb G$, we consider a family~$\mathbb X= \left \{ X_1,...,X_k \right \}$
of left-invariant vector fields on $\mathbb G$ satisfying the H\"ormander condition, i.e. $\mathcal G$ is the Lie algebra generated by the $X_i's$. In the following, although not stated, all the functional spaces depend on the field $\mathbb X$. 

\noindent A standard metric on~$\mathbb G$, called the Carnot-Caratheodory metric, is naturally associated with $\mathbb X$ and is defined as follows: let $\ell : [0,1] \to \G$ be an absolutely continuous path. We say that $\ell$ is admissible if there exist measurable functions $c_1,...,c_k : [0,1] \to \mathbb C$ such that, for almost every $t \in [0,1]$, one has~$\displaystyle\ell'(t)=\sum_{i=1}^k c_i(t) X_i(\ell(t)).$
If~$\ell$  is admissible, its length is defined by~$\displaystyle |\ell |= \int_0^1\big(\sum_{i=1}^k |c_i(t)|^2 \,dt \big)^{ \frac 12}$.
 For all $x,y \in \G $, define~$d(x,y)$ as  the infimum of the lengths  of all admissible paths joining $x$ to $y$ (such a curve exists by the H\"ormander condition). This distance is left-invariant. For short, we denote by~$|x|$ the distance between $e$, the neutral element of the group, and~$x$  so that the distance from $x$ to~$y$ is equal to  $|y^{-1}x|$. 
For all $r>0$, denote by $B(x,r)$ the open ball in $\G$ with respect to the Carnot-Caratheodory distance and by $V(r)$ the Haar measure of any ball. There exists $d\in \N^{\ast}$ (called the local dimension of $(G,\mathbb X)$) and~$0<c<C$ such that, for all $r\in ]0,1[$,
$$
cr^d\leq V(r)\leq Cr^d,
$$
see \cite{nsw}. When $r>1$, two situations may occur (see \cite{guivarch}): 
\begin{itemize}
\item Either there exist $c,C,D >0$ such that, for all $r>1$, 
$c r^D \leq V(r) \leq C r^D$
where~$D$ is called the dimension at infinity of the group (note that, contrary to $d$, $D$ does not depend on $\mathbb X$). The group is said to have polynomial volume growth. 
\item Or  there exist $c_1,c_2,C_1,C_2 >0$ such that, for all $r>1$, 
$c_1 e^{c_2r} \leq V(r) \leq C_1 e^{C_2r}$
and the group is said to have exponential volume growth. 
\end{itemize} 
When $\G$ has polynomial volume growth, it is plain to see that there exists a constant~$C>0$ such that  for all $r>0$, $V(2r)\leq CV(r)$.  In turn this implies that there exist $C>0$ and $\kappa>0$ such that  for all~$r>0$ and all $\theta>1$,
$V(\theta r)\leq C\theta^{\kappa}V(r).$

\noindent We denote $\Delta_\G  =\displaystyle \sum_{i=1}^k X_i^2$
the sub-laplacian on $\G$.

\subsection{Nilpotent Lie groups} A Lie group is said to be   {\it nilpotent}  if  its Lie algebra~${\mathcal G}$ is nilpotent: more precisely writing~${\mathcal G}^1 = {\mathcal G}$ and defining inductively~¬ ${\mathcal G}^{k+1} = [{\mathcal G}^k,{\mathcal G}^k]$, then there is~$n$ such that~${\mathcal G}^n = \{0\}$. It can be shown that such groups are always of polynomial growth (see for instance~\cite{dungey}).
 
\subsection{Stratified (Carnot) and H-type groups}\label{carnot}  Stratified groups are a particular version of nilpotent  groups, which admit a {\it stratified} structure and for which~$V(r)\sim r^Q$ for some positive~$Q$, for all~$r>0$. One advantage of this additional structure is that such groups admit dilations.   Important examples of such groups are H-type groups, a particular example being the Heisenberg group.

\noindent   More precisely, a stratified (or Carnot) Lie group $\G$ is  simply connected and its Lie algebra admits a stratification, i.e. there exist linear subspaces $V_1,...,V_r$ of $\mathcal G$ such that 
$\mathcal G= V_1 \oplus ... \oplus V_r$
which satisfy~$[V_1,V_i]=V_{i+1}$
for $i=1,...,r-1$ and~$[V_1,V_r]=0$. By~$[V_1,V_i]$ we mean the subspace of~$\mathcal G$ generated by the elements $[X,Y]$ where~$X \in V_1$ and $Y \in V_i$. Carnot groups are nilpotent. Furthermore, via the exponential map,~$\mathbb G$ and~$\mathcal G$ can be identified as manifolds. The dilations $\gamma_\delta$ ($\delta >0$) are then defined (on the Lie algebra level) by 
$$\gamma_{\delta}(x_1+...+x_r)=\delta x_1+ \delta^2x_2+...+\delta^r x_r,\,\,\,\,x_i \in V_i. $$ 
 We   define the homogeneous dimension $Q=\mbox{dim} V_1+ 2  \mbox{dim} V_2+\dots+r  \mbox{dim} V_r. $
If $\G$ is a Carnot group, we have for all $r>0$, $V(r)\sim r^Q$ (see \cite{FS}).  We shall say that the~$Q-$dimensional Carnot group is of step~$r$: for instance the Heisenberg group ${\mathcal H}^d$ is a Carnot group and~$Q=2d+2$.

\noindent The previous abstract definition of Carnot groups is not always very practical. It is however possible to prove (see \cite{bonUgu}) that any $N $-dimensional Carnot group of step~$2$ with~$m$ generators is isomorphic to~$(\mathbb R^N, \circ)$ with the law given by ($N=m+n$, $x^{(1)} \in \mathbb R^m,x^{(2)} \in \mathbb R^n$)
$$
(x^{(1)},x^{(2)}) \circ (y^{(1)},y^{(2)})=\begin{pmatrix}
x_j^{(1)}+y_j^{(1)},\,\,\,j=1,...,m \\x_j^{(2)}+y_j^{(2)}+\frac12 \langle x^{(1)}, U^{(j)}y^{(1)} \rangle,\,\,\,j=1,...,n
\end{pmatrix},
$$
where $U^{(j)}$ are $m \times m$ linearly independent skew-symmetric matrices.

\noindent  With this at hand, one can give the definition of a group of Heisenberg-type (H-type henceforth).  These groups are two-step stratified nilpotent Lie groups whose Lie algebra carries a suitably compatible inner product, see \cite{kaplan1}. One of these groups is the nilpotent Iwasawa subgroup of semi-simple Lie groups of split rank one (see~\cite{koranyi2}). 
More precisely, an H-type group is a Carnot group of step $2$ with the following property: the Lie algebra~$\mathcal G $ of $\G$ is endowed with an inner product~$\langle \cdot , \cdot\rangle$ such that if $\mathcal Z$ is the center of~$\mathcal G$, then~$[\mathcal Z^\perp,\mathcal Z^\perp]=\mathcal Z$
and moreover for every $z \in \mathcal Z$, the map $J_z: \mathcal Z^\perp \to \mathcal Z^\perp$ defined by~$
\langle J_z (v), w \rangle= \langle z, [v,w] \rangle
$
for every $w \in \mathcal Z^\perp$ is an orthogonal map whenever $\langle z, z \rangle =1$.
  If~$m= $ dim $(\mathcal Z^\perp)$ and $n=$ dim $(\mathcal Z)$, then any H-type group is canonically isomorphic to~$\mathbb R^{m+n}$ with the above group law, where the matrices~$U^{(j)}$   satisfy  the additional property~$U^{(r)}U^{(s)}+U^{(s)}U^{(r)}=0$
for every $r,s \in \left \{1,...,n \right \}$ with $r \neq s$. Whenever the center of the group is one-dimensional, the group is canonically isomorphic to the Heisenberg group on $\mathbb R^{m+1}$. We shall always identify~$\mathcal Z^\perp$ with $\C^\ell$ with~$2\ell=m$ and~$\mathcal Z$ to~$\R^n$ thanks to the discussion  above.  Note that the homogeneous dimension of a H-type group so defined is~$Q = 2 \ell + n$. On an H-type group $\G$, the vector-fields in the algebra $\mathcal G$ are given by 
$$X_j=\frac{\partial }{\partial x_j}+ \frac12 \sum_{k=1}^n \sum_{l=1}^{2\ell} z_l U^{(k)}_{l,j} \frac{\partial}{\partial t_k} \quad \mbox{and} \quad Y_j=\frac{\partial }{\partial y_j}+ \frac12 \sum_{k=1}^n \sum_{l=1}^{2\ell} z_l U^{(k)}_{l,j+\ell} \frac{\partial}{\partial t_k}$$
for $j=1,...,\ell$, $z=(x,y) \in \R^{2¬¨¬®‚Äö√Ñ‚Ä \ell}$ and $t \in \R^n$. 
In the following we shall denote by~${\mathcal X} $ any element of the family~$(X_1,\dots,X_\ell,Y_1,\dots,Y_\ell)$.
  The hypo-elliptic Kohn Laplacian on H-type groups writes  
$$\Delta_\G= \sum_{j=1}^m \frac{\partial^2}{\partial x^2_j}+\frac14 |x|^2 \sum_{s=1}^n \frac{\partial^2}{\partial t^2_s}+ \sum_{s=1}^n \sum_{i,j=1}^m x_i U^{(s)}_{ij} \frac{\partial^2}{\partial t_s \partial x_j} \cdotp $$ 

\subsection{Main results and structure of the paper}
In  \cite{crt}, the authors investigate  the algebra properties of the Bessel space 
$$L^p_\alpha(\mathbb G) = \bigl\{ f \in L^p (\mathbb G), \: (-\Delta_\G)^\frac\alpha 2f \in L^p (\mathbb G) \bigr\}
$$
and their homogeneous counterpart,
where $\mathbb G$ is any unimodular Lie group. 

\medskip
\noindent   Our first theorem  considers   Besov spaces in the general setting of groups with polynomial volume growth. The case~$s \in (0,1)$ is obtained, both for inhomogeneous and homogeneous spaces, by using an equivalent definition in terms of differences (see~\cite{saloffcoste}). The general case is only proved in the case of inhomogeneous  spaces and uses the fact that local Riesz transforms are continuous in~$L^p$ for~$1<p<\infty$ (whence the restriction on~$p$ below), along with an interpolation argument to obtain all values of~$s$. 
\begin{theo}\label{main}
Let $\G$ be a Lie group with polynomial volume growth.

\medskip

\noindent
 For every~$s \in (0,1)$ and~$1 \leq p,q \leq \infty$, the spaces~$B^s_{p,q}(\G) \cap L^\infty(\G)$ and~$\dot B^s_{p,q}(\G) \cap L^\infty(\G)$ are algebras under pointwise multiplication. 

\smallskip

\noindent The same property holds if~$s \geq 1$ for~$B^s_{p,q}(\G) \cap L^\infty(\G)$, with the additional restriction that~$1< p<\infty$.
\end{theo}\begin{remark}
We shall   give  a generalization of Theorem~{\rm\ref{main}}   to the case when the space~$L^\infty(\G) $ is replaced by~$L^r(\G) $ (see Propositions~{\rm\ref{moregeneralresult}} and~{\rm\ref{moregeneralresult2}}).
\end{remark}

\noindent
One can recover the full range of indexes~$p$, as well as   homogeneous Besov spaces, in the context of H-type groups thanks to the paraproduct algorithm. Before stating the result let us give an intermediate statement in the case of nilpotent groups.  Its proof requires   
the continuity of  Riesz transforms, as well as 
a result which is to our knowledge new even in the context of the Heisenberg group (see Proposition~\ref{bs+1bs}) and which links~$\dot  B^s_{p,q}(\G) $ and~$\dot B^{s+1}_{p,q}(\G) $ in terms of the action of~$X_i$ and not only powers of the sublaplacian. 
\begin{theo}\label{mainnil} 
Let $\G$ be a  nilpotent Lie group.  

\medskip

\noindent
For every $1 \leq s< d$, the space~$\dot B^s_{\frac ds,1} (\G)$ is embedded in~$L^\infty (\G)$ and is an algebra. 

\smallskip
\noindent Moreover for every~$1 < s$ and every~$1<p <\infty$,  if~$f$ and~$g$ belong to the space~$ \dot B^s_{p,\frac{s-1}s  } \cap L^{\infty}(\G) $  then~$fg $ belongs to~$ \dot B^s_{p,1} \cap L^{\infty}(\G)$.

\smallskip
 \noindent Finally if~$1<p_1,p_2<\infty$ with~$1/p = 1/p_1+1/p_2$, if~$1 \leq q \leq \infty$, and if~$f $ belongs to~$\dot B^s_{p_1,q  } \cap L^{p_1}(\G) $ and~$g$ belongs to~$\dot B^s_{p_2,q  } \cap L^{p_2}(\G) $ then~$fg \in \dot B^s_{p,q}\cap L^p (\G)$, for any~$s >0$. 
 \end{theo}
\begin{remark}
Unfortunately we are unable to recover, in the case of nilpotent groups, the full algebra property  due to the (technical) fact that Besov spaces do not interpolate very well when the integrability indexes are different. The second property in Theorem~{\rm\ref{mainnil}}  is almost an algebra property, except for   a   loss in the third (summation) index. As to the last property this time the integrability index is changed in the product. The reason for those losses will appear clearly in the proof of the   theorem.
\end{remark}

\noindent Finally, in the context of H-type groups, thanks to paraproduct techniques, one can enlarge the range of admissible spaces and prove the following result.
\begin{theo}\label{mainHtype}
Let $\G$ be an H-type group.  For every $s >0$ and~$1 \leq p,q \leq \infty$,   the spaces~$B^s_{p,q}(\G) \cap L^\infty(\G)$ and $\dot B^s_{p,q}(\G) \cap L^\infty(\G)$ are   algebras under pointwise multiplication.
\end{theo}

\noindent  Besov spaces are defined in the coming section, and Theorems~\ref{main} and~\ref{mainnil}   are proved in Sections~\ref{proofthmmain} and~\ref{proofthmmainnil} respectively.
 We   present the proof of Theorem~\ref{mainHtype} in Section~\ref{paradiff}.

\medskip

\noindent We shall write~$A \lesssim B$ if there is a universal constant~$C$ such that~$A \leq CB$. Similarly we shall write~$A \sim B$ if~$A \lesssim B$ and~$B \lesssim A$.

\bigskip

\noindent {\bf Acknowledgements. } The authors are very grateful to L. Saloff-Coste for comments on a previous version of this text. They also thank the anonymous referee for questions and suggestions which improved the presentation.

\section{Littlewood-Paley decomposition on   groups of polynomial growth, and Besov spaces}\label{lpbesov}
\setcounter{equation}{0}

This section is devoted to a presentation of the Littlewood-Paley decomposition on  groups of polynomial growth, together with some standard applications. A general approach to the Littlewood-Paley decomposition on Lie groups with polynomial growth is investigated in~\cite{furioli}. We also refer to~\cite{BGX1} or~\cite{bg} for the case of the Heisenberg group. We recall here the construction of the homogeneous and inhomogeneous decompositions. For details and proofs of the results presented in this section we refer to~\cite{chamorro}, \cite{furioli} and~\cite{hulanicki}.

\subsection{Littlewood-Paley decomposition}
We first review the dyadic decomposition constructed in~\cite{furioli}. Let $\chi \in C^{\infty}(\R)$ be an even function such that~$0 \leq \chi \leq 1$ and~$\chi =1$ on $[0,1/4]$, $\chi =0$ on~$[1,\infty[$. Define~$\psi(x)= \chi(x/4)-\chi(x)$,  so that the support of~$ \psi $ is included in~$ [-4, -1/4]\cup [1/4, 4].$ The following holds:
$$
 \forall  \tau \in \R^*,   \quad \sum_{j \in \Z}
 \psi(2^{-2j}\tau)  =  1 
\quad \mbox{and} \quad
 \chi (\tau) + \sum_{j \geq 0} \psi(2^{-2j}\tau) = 1, 
\quad  \forall \tau \in 
\R.
$$
Introduce the spectral decomposition of the hypo-elliptic Laplacian
$$-\Delta_{\mathbb G} = \int_0^\infty \lambda dE_\lambda .$$
Then we have
$$\chi (-\Delta_{\mathbb G})=\int_0^\infty \chi(\lambda) dE_\lambda \quad \mbox{and} \quad \psi(-2^{-2j} \Delta_{\mathbb G})=\int_0^\infty \psi(2^{-2j}\lambda ) dE_\lambda . $$
We then define for $j \in \mathbb N$ the operators 
$$S_0 f= \chi(-\Delta_{\mathbb G})f \quad \mbox{and} \quad\Delta_j f= \psi(-2^{-2j} \Delta_{\mathbb G}) f. $$
The homogeneous  Littlewood-Paley decomposition of $f$ in $\mathcal S'(\mathbb G)$ is~$f= \sum_{j \in \Z} \Delta_j f,  $
while the inhomogeneous  one is~$f= S_0f + \sum_{j=0}^\infty \Delta_j f . $ 
\begin{theo}[\cite{furioli}]
Let $\mathbb G$ be a Lie group with polynomial growth and~$p \in (1,\infty)$. Then~$u$ belongs to~$  L^p(\mathbb G)$ if and only if $S_0 u$ and~$\sqrt{\sum_{j=0}^\infty |\Delta_j u |^2}$ belong  to~$L^p(\mathbb  G)$. Moreover, we have  
$\displaystyle \|u\|_{L^p(\mathbb G)} \sim \|S_0u\|_{L^p(\mathbb G)}+\Bigl \|\bigl({\sum_{j=0}^\infty |\Delta_j u |^2}\bigr)^\frac12\Bigr \|_{L^p(\mathbb G)}. $
\end{theo}
\noindent In the following we shall denote by $\Psi_j$ the kernel of the operator~$\psi(-2^{-2j} \Delta_{\mathbb G})$. One can show that~$\Psi_j$ is mean free (see Corollary 5.1 of~\cite{chamorropaper} for Carnot groups, and Theorem~7.1.2 of~\cite{chamorro} for an extension to groups of polynomial growth). In the context of Carnot groups,~$\Psi_j$ satisfies the dilation property \noindent 
$$
\Psi_j(x)= 2^{Qj}\Psi_0(2^j x). 
$$
In the more general context of groups of polynomial growth, this does not hold but one has nevertheless the following important estimates: let~$\alpha \in \N$ and~$\displaystyle I \in \bigcup_{\beta \in \N}\{1,\dots,k\}^\beta$ be given, as well as~$p \in [1,\infty]$. The following result is due to~\cite{furioli}:
\begin{equation}\label{estimatekernel}
\forall j \geq 0, \quad \| (1+|\cdot|)^\alpha X^I \Psi_j \|_{L^p(\G)} \lesssim 2^{j(\frac d{p'} + |I|)},
\end{equation}
where~$1/p + 1/p' = 1$. We have denoted~$X^I = X_{i_1} \dots X_{i_\beta}$ and~$|I| = \beta$.
 Moreover as proved in~\cite{chamorro}, Theorem 7.1.2, one has
 \begin{equation}\label{secondestimatekernel}
\forall j  \in \Z, \quad \|   X_i  \Psi_j \|_{L^1(\G)} \lesssim 2^{j }.
\end{equation}
 Finally putting together   classical estimates on the heat kernel (see~\cite{crt} or~\cite{coulhonscv} for instance) and the  methods of~\cite{furioli} allows to write that for any~$\alpha \geq 0$,
 \begin{equation}\label{thirdestimatekernel}
\forall j  \in \Z, \quad \bigl\|  |\cdot|^\alpha  \Psi_j \bigr\|_{L^1(\G)}   \lesssim 2^{j \alpha}.
\end{equation}

 \subsection{Besov spaces}

As a standard application of the Littlewood-Paley decomposition, one can define (inhomogeneous) Besov spaces on Lie groups with polynomial volume growth in the following way: let $s \in \R$ and $1 \leq p\leq +\infty$ and~$0 < q \leq \infty$, then $B^s_{p,q}(\mathbb G) $ is the space
$$
  \left \{ f \in \mathcal S'(\mathbb G), \,\|f\|_{  B^s_{p,q}(\mathbb G) } =  \|S_0f\|_{L^p(\mathbb G)}+ \Big ( \sum_{j=0}^\infty (2^{js} \|\Delta_jf \|_{L^p(\mathbb G)})^q\Big )^{1/q}< \infty \right \}$$
with the obvious adaptation if~$q = \infty$. When~$p=q=2$ one recovers the usual Sobolev spaces (see for instance~\cite{bg} for a proof in the case of the Heisenberg group). Note that when~$s>0$ one sees easily that~$\|S_0f\|_{L^p(\mathbb G)}$ may be replaced by~$\| f\|_{L^p(\mathbb G)}$. 
Using   Bernstein inequalities (Proposition~4.2 of~\cite{furioli}) one gets immediately that if~$0<s$ then
\begin{equation}\label{besovembedding}
p_1 \leq p_2 \Longrightarrow B^{s+\frac d{p_1}-\frac d{p_2}}_{p_1,q} \cap L^{p_2} \hookrightarrow B^{s}_{p_2,q} \cap L^{p_1}
\end{equation}
where recall that~$d$ is the local dimension of~$\G$.

\medskip
\noindent One can also define the homogeneous counterpart of the above norm: 
$$\|f\|_{\dot B^s_{p,q}(\mathbb G) } = \Big ( \sum_{j \in \Z} (2^{js} \|\Delta_j f\|_{L^p(\mathbb G)})^q\Big )^{1/q} 
$$
but proving that this does provide a (quasi)-Banach space is not an easy matter, and is actually not true in general, even in the euclidean case (see for instance~\cite[Chapter~2]{BCD} for comments on that subject). To recover a Banach space in the context of Carnot groups, the homogeneous space $\dot B^s_{p,q}(\mathbb G)$ can be defined as     the set of functions in $\mathcal S'(\mathbb G)$ modulo polynomials, such that the above norm is finite  (see~\cite{fuhr}). In the present study however this will not be an issue, even if the group is not stratified: we define~$\dot B^s_{p,q}(\mathbb G) $  as the completion   for the above norm of the set of smooth functions such that~$\Delta_j f \to 0$ as~$j \to -\infty$, and we shall always be considering the intersection of~$\dot B^s_{p,q}(\mathbb G) $    with a Banach space (such as~$L^\infty$).

\smallskip
\noindent Note that Bernstein inequalities imply as in~(\ref{besovembedding}) that
$$
p_1 \leq p_2 \Longrightarrow\dot B^{s+\frac d{p_1}-\frac d{p_2}}_{p_1,q}\hookrightarrow\dot B^{s}_{p_2,q} .
$$

\medskip
\noindent 
    Besov spaces are often rather defined using the heat flow (the advantage being that it does not require the Littlewood-Paley machinery). 
In \cite{furioli}, the authors prove that if~$s \in \R$, then $f \in B^s_{p,q}(\mathbb G)$ is equivalent to: for all~$t >0$, the function~$e^{t\Delta_{\mathbb G}}f$ belongs to~$L^p(\mathbb G)$ and 
\begin{equation}\label{charheat}
\Big (\int_0^1 t^{-sq/2}\|(t(-\Delta_\G))^{m/2}e^{t\Delta_\G }f\|^q_{L^p(\mathbb G)}\frac{dt}{t} \Big )^{1/q} < \infty
\end{equation} for $m \geq 0$ greater than~$s$. 
We shall not be using this characterization here.

\section{Proof of Theorem \ref{main}}\label{proofthmmain}
\setcounter{equation}{0}
\subsection{The case $s \in (0,1)$}
We start by dealing with the case~$s \in (0,1)$, and  use an idea of~\cite{crt}   which consists in representing the norm on Besov spaces by suitable functionals. More precisely, we introduce the following  functional (note that it differs slightly from that used in \cite{crt}), writing $\tau_{w}f(w')=f(w'w)$:
$$\mathcal S_{s,p} f(w)=  \frac{\|\tau_{w}f-f\|_{L^p(\G)}}{|w|^s}   \cdotp$$
\begin{pro}\label{crucial}
Let $\mathbb G$ be a  Lie group with polynomial growth. Then for any $s \in (0,1)$ and $p,q \in [1,+\infty]$, we have 
$$\|f\|_{  B^s_{p,q}(\G)} \sim \|f\|_{L^p} +  \|\mathcal S_{s,p} f\|_{L^q(\G, \frac{{\mathbf 1}_{|y| \leq 1 } dy}{V(|y|)})}
$$
and
$$
\|f\|_{\dot  B^s_{p,q}(\G)} \sim   \|\mathcal S_{s,p} f\|_{L^q(\G, \frac{ dy}{V(|y|)})} .
$$
\end{pro}

\noindent Once Proposition \ref{crucial} is proved, the algebra property follows immediately in the case when~$s \in (0,1)$. Indeed, let $f,g$ belong to the space~$ (  B^s_{p,q} \cap L^{\infty}) (\mathbb G)$ for $s \in (0,1)$. It is easy to see that 
\begin{equation}\label{holderlinfty}
\mathcal S_{s,p}(fg) \leq \|f\|_{L^\infty} \: \mathcal S_{s,p} g+ \|g\|_{L^\infty}  \:\mathcal S_{s,p} f, 
\end{equation}
hence the result using the equivalence in Proposition \ref{crucial}.  The same holds in the homogeneous case.

\begin{remark}\label{holdergeneral}
One can   extend~(\ref{holderlinfty})   to the following, with~$1/a_i + 1/b_i = 1/p$:
$$
\mathcal S_{s,p}(fg) \leq    \|f\|_{L^{a_1}}   \mathcal S_{s,b_1} g + \|g\|_{L^{a_2}} \mathcal S_{s,b_2}  f\ .$$
\end{remark}
\medskip

\noindent We now prove Proposition \ref{crucial}. Note that 
this result was already proved in~\cite{saloffcoste} using the characterization~(\ref{charheat}). We choose to present a proof using the Littlewood-Paley definition here, which is inspired by the proof of the euclidean case in~\cite{BCD} for instance.
We need to prove that for $s\in (0,1)$
$$
\sum_{j \in \Z} (2^{js} \|\Delta_j f\|_{L^p(\mathbb G)})^q \sim \|f\|_{L^p} ^q+ \int_{\mathbb G} {\mathbf 1}_{|w| \leq 1 } \frac{\|\tau_{w}f-f\|^q_{L^p(\mathbb G)} }{V(|w|)|w|^{sq}}dw 
$$
with the obvious modification if~$q = \infty$.
 Compared to the euclidean case, one is missing the usual dilation property, which will be replaced by estimate~(\ref{estimatekernel}). The classical proof also uses a Taylor expansion at order one, which we must adapt to our context in order to use only horizontal vector fields (which alone appear in~(\ref{estimatekernel})). 
 Let us start by  noticing that
 $$
  \|f\|_{L^p}^q \leq \sum_{j \leq 0} (2^{js} \|\Delta_j f\|_{L^p(\mathbb G)})^q  \, .
 $$ 
 Next let us bound the quantity~$\|\tau_{w} \Delta_j f-\Delta_j f \|_{L^p(\G)}$ .
Recalling that~$\displaystyle \Delta_j = \sum_{| j'-j | \leq 1} \Delta_j \Delta_{j'},$
we have
$$
\tau_{w} \Delta_j f-\Delta_j f =   \sum_{| j'-j | \leq 1} \Delta_{j'} f \star \bigl( \tau_{w} \Psi_j-\Psi_j \bigr),
$$
where~$\Psi_j$ is the kernel associated with~$\psi(2^{-2j} \Delta_{\mathbb G})$. It follows by Young's inequality that
$$
\|\tau_{w} \Delta_j f-\Delta_j f\|_{L^p(\G)} \leq   \sum_{| j'-j | \leq 1} \| \Delta_{j'} f\|_{L^p} \| \tau_{w} \Psi_j-\Psi_j \|_{L^1}.
$$
Now let us estimate~$\| \tau_{w} \Psi_j-\Psi_j \|_{L^1}$. We have
\begin{eqnarray*}
 ( \tau_{w} \Psi_j-\Psi_j )(x)&=& \int_0^1 \frac d{ds} \Psi_j(x \varphi (s)) \: ds \\
&=& \sum_{\ell = 1}^k  \int_0^1 c_\ell (s)  (X_\ell (x \varphi (s))  \Psi_j)(x \varphi (s))  \: ds ,
\end{eqnarray*}
where~$\varphi$ is an admissible path linking~$e$ to~$w$. It follows that
\begin{eqnarray*}
 \|\tau_{w} \Psi_j - \Psi_j\|_{L^1} &\leq &\int_\G \sum_{\ell = 1}^k  \int_0^1 | c_\ell (s) | \:  \bigl| (X_\ell (x \varphi (s))  \Psi_j)(x \varphi (s)) \bigr|  \: ds dx \\
 &\leq &\sum_{\ell = 1}^k\int_0^1 | c_\ell (s) |  \: ds \:\|X_\ell \Psi_j\|_{L^1} 
 \end{eqnarray*}
 by the Fubini theorem and a change of variables. Using~(\ref{secondestimatekernel})
we get
$$
\forall j\in \N, \quad   \|\tau_{w} \Psi_j - \Psi_j\|_{L^1} \lesssim 2^{j} \sum_{\ell = 1}^k\int_0^1 | c_\ell (s) | \: ds
  $$
  so by definition of~$|w|$ and by the Cauchy-Schwarz inequality we find  
  $$
  \forall j \in \N, \quad  \|\tau_{w} \Psi_j - \Psi_j\|_{L^1} \lesssim 2^{j}  |w|.
  $$
 This  implies that there is a sequence~$(c_j) $ in the unit ball of~$ \ell^q$ such that
  \begin{equation}\label{first}
\forall j\in \N, \quad  \|\tau_{w} \Delta_j f-\Delta_j f \|_{L^p(\G)} \lesssim c_j |w| 2^{j(1-s)} \|f\|_{  B^s_{p,q}(\G)}.
\end{equation}
 On the other hand  one has of course
\begin{equation}\label{second}
 \|\tau_{w} \Delta_j f-\Delta_j f \|_{L^p(\G)} \lesssim c_j 2^{-js} \|f\|_{  B^s_{p,q}(\G)}.
\end{equation}
Now let us define~$j_w \in \Z$ such that $\frac{1}{|w|} \leq 2^{j_w} \leq \frac{2}{|w|}\cdotp$ 
Then
using~(\ref{first}) for low frequencies   and~(\ref{second}) for high frequencies allows to write
$$
\|\tau_{w}f-f\|_{L^p(\G)} \lesssim  \|f\|_{  B^s_{p,q}(\G)} \left(
\sum_{j \leq j_w}c_j 2^{j(1-s)}|w| + \sum_{j > j_w}c_j 2^{ -js}
\right).
$$
Let us first consider the case~$q = \infty$. Then one finds directly that
$$
\|\tau_{w}f-f\|_{L^p(\G)} \lesssim |w|^s  \|f\|_{  B^s_{p,q}(\G)}
$$
which proves one side of the equivalence. The case~$q < \infty$ is slightly more technical but is very close to the euclidean case. We include it here for sake of completeness. 
We have that 
$$\Big \|\frac{\|\tau_{w}f -f\|_{L^p}}{|w|^s}\Big\|^q_{L^q(\G,\frac{ {\mathbf 1}_{|w| \leq 1 } }{V(|w|)})} \lesssim 2^q \|f\|_{\dot B^s_{p,q}}^q (I_1+I_2)$$
where 
$$I_1=\int_\G {\mathbf 1}_{|w| \leq 1 }  \Big ( \sum_{j \leq j_w} c_j 2^{j(1-s)} \Big )^q \frac{|w|^{q(1-s)}dw}{V(|w|)} \quad \mbox{and}$$ 
$$I_2=\int_\G  {\mathbf 1}_{|w| \leq 1 } \Big ( \sum_{j > j_w} c_j 2^{-js} \Big )^q \frac{|w|^{-qs}dw}{V(|w|)} \cdotp $$
H\"older's inequality with the weight $2^{j(1-s)}$ and the definition of $j_w$ imply   
$$\Big (\sum_{j \leq j_w} c_{j}2^{j(1-s)} \Big )^q \lesssim |w|^{-(1-s)(q-1)} \sum_{j \leq j_w} c^q_{j}2^{j(1-s)}. $$
By Fubini's theorem, we deduce that 
$$I_1 \lesssim \sum_{j \in \N}\int_{B(0,2^{-j+1})} |w|^{1-s}\frac{dw}{V(|w|)} 2^{j(1-s)}c_j^q \lesssim 1, $$
since~$\|(c_j)\|_{\ell^q} \leq 1$. 
The estimate on $I_2$ is very similar.  Note that it is crucial here that~$s \in (0,1)$.

\noindent The converse inequality is easy to prove and only depends on the fact that the mean value of~$\Psi_j$ is zero. We write indeed   
$$
  \Delta_j f (w) = \int \tau_v f(w) \Psi_j (v) \: dv   =  \int( \tau_v f(w) - f(w)) \Psi_j (v) \: dv 
$$
 so that  
 $$
 2^{js} \|  \Delta_j f \|_{L^p} \leq \sup_{v \in \G} \frac {\| \tau_v f-f\|_{L^p}} { |v|^s}  \int 2^{js} |v|^s |\Psi_j (v) | \: dv 
.
 $$
Then~(\ref{thirdestimatekernel})  implies that
$$
 2^{js} \|  \Delta_j f \|_{L^p} \leq \sup_{v \in \G} \frac {\| \tau_v f-f\|_{L^p}} { |v|^s}  \, \cdotp
$$
Since
$$
 \sup_{|v| \geq 1  } \frac {\| \tau_v f-f\|_{L^p}} { |v|^s} \leq  2 \|f\|_{L^p}  \, ,
$$
we get finally
$$
\sup_{j  \in \Z}  2^{js} \|  \Delta_j f \|_{L^p} \leq \sup_{|v| \leq 1 } \frac {\| \tau_v f-f\|_{L^p}} { |v|^s} + 2  \|f\|_{L^p} \,  ,
$$
the result follows in the case~$q = \infty$.
The case~$q < \infty$ is similar though a little more technical, as above. 

\noindent The homogeneous case is dealt with in a similar fashion. We leave the details to the reader. 
 This proves Proposition \ref{crucial}.  \qed

\medskip

\noindent Using Remark~\ref{holdergeneral}, the same proof provides the following result, which will be useful in the next section.
  \begin{pro}\label{moregeneralresult}
 Let $\G$ be a Lie group with polynomial volume growth. 
 
 \noindent For every~$0 < s < 1$ and $1 \leq p,q \leq \infty$ one has, writing~$1/p = 1/a_i + 1/b_i$
 $$
 \|fg\|_{B^s_{p,q}} \leq \|f\|_{L^{a_1}} \|g\|_{B^s_{b_1,q}} +  \|g\|_{L^{a_2}} \|f\|_{B^s_{b_2,q}}
  $$ 
 and
 $$
 \|fg\|_{\dot B^s_{p,q}} \leq \|f\|_{L^{a_1}} \|g\|_{\dot B^s_{b_1,q}} +  \|g\|_{L^{a_2}} \|f\|_{\dot B^s_{b_2,q}}.
 $$ \end{pro}
  
 \subsection{The case $s\geq 1$ (inhomogeneous spaces)}
 We shall first deal with the case when~$s$ is not an integer. We use the well-known fact that the ``local Riesz transforms"~$({\rm Id}-\Delta_\G)^{\frac {m-1}2} X_i  ({\rm Id}-\Delta_\G)^{-\frac m2 }$ are bounded over~$L^p(\G)$  for~$1<p<\infty$ (see for instance~\cite{dungey}).
 This implies easily (see the next section where the same result is proved in the more difficult homogeneous case) that
 $$
 f \in B^{s+1}_{p,q} \iff f \in B^{s}_{p,q} \quad \mbox{and} \quad X_i f \in B^{s}_{p,q} \quad \forall i =1,\dots,k.
 $$ 
We can then  follow the lines of~\cite{crt}, by writing 
$
\|fg\|_{B^{s+1}_{p,q}} \sim \|fg\|_{B^s_{p,q}} + \sum_{i = 1}^k \|X_i (fg)\|_{B^s_{p,q}} 
$
and by arguing by induction: let us detail the case~$s = 1+s'$ with~$0<s'<1$.
 On the one hand we know that  for all~$1 \leq p,q \leq \infty$   and if~$1/a_i + 1/b_i = 1/p$,
$$
\|fg\|_{B^{s'}_{p,q}} \lesssim \|f\|_{L^{a_1}} \| g\|_{B^{s'}_{b_1,q}} +   \|g\|_{L^{a_2}} \| f\|_{B^{s'}_{b_2,q}} .
$$
Then we write, by the Leibniz rule,
$$
 \|X_i (fg)\|_{B^{s'}_{p,q}}  \leq  \|f X_i g\|_{B^{s'}_{p,q}}  +  \|gX_i f\|_{B^{s'}_{p,q}} 
$$
and we have,   by Proposition~\ref{moregeneralresult},
\begin{equation}\label{estimatefXig}
 \|f X_i g\|_{B^{s'}_{p,q}}  \lesssim   \|f\|_{L^{a_1}}  \|X_i g\|_{B^{s'}_{b_1,q}} + \| f\|_{B^{s'}_{a_2,q}}  \|X_ig\|_{L^{b_2}} .
\end{equation}
The estimate on~$gX_i f$ in~$B^{s'}_{p,q}$ is similar so we shall not write the details for that term.

\noindent
The first term on the right-hand side of~(\ref{estimatefXig}) is very easy to estimate since
$$
 \|X_i g\|_{B^{s'}_{b_1,q}} \lesssim \|g\|_{B^{s}_{b_1,q}}.
$$
So let us turn to the second term.
Let us first  estimate~$f$ in~$B^{s'}_{a_2,q}$. We have clearly, since~$s'\leq s$,
$$
 \| f\|_{B^{s'}_{a_2,q}}  \lesssim \|f\|_{B^{s}_{a_2,q}} .
$$
Now let us turn to the estimate of~$X_i g$ in~$L^{b_2}$, choosing~$1<b_2<\infty $. We   use the fact that
$$
\begin{aligned}
\|X_i g\|_{L^{b_2}}& \lesssim \|X_i ({\rm Id}-\Delta)^{-\frac12} ({\rm Id}-\Delta)^{\frac12} g \|_{L^{b_2}}  \\
& \lesssim \| ({\rm Id}-\Delta)^{\frac12} g \|_{L^{b_2}}
\end{aligned}
$$
by the continuity of the local Riesz transforms.  Since~$[ \Delta_j, \Delta_\G] = 0$,  Bernstein's lemma (see Proposition~4.3 of~\cite{furioli}) implies  
 $$
 \|\Delta_j ({\rm Id}-\Delta)^{\frac12} g \|_{L^{b_2}}  \lesssim 2^{ j} \|\Delta_j   g\|_{L^{b_2}}.
$$
This implies that
$$
\begin{aligned}
  \| ({\rm Id}-\Delta)^{\frac12} g \|_{L^{b_2}}   & \lesssim \| S_0({\rm Id}-\Delta)^{\frac12} g\|_{L^{b_2}} + \sum_{j \geq 0} 2^j \|\Delta_j g\|_{L^{b_2}} \\
 & \leq \|g\|_{L^{b_2}} +\sum_{j \geq 0} 2^{js} \|\Delta_j g\|_{L^{b_2}} 2^{j {(1-s)} }\\
 & \lesssim \|g\|_{  B^{s}_{{b_2},q}}
\end{aligned}
$$
since~$s>1$.
This gives the required estimate for the second term in~(\ref{estimatefXig}) and that allows to conclude the proof of Theorem~\ref{main} in the case when~$s \in \R^+ \setminus \N$. The general case~$s>0$ is then obtained by interpolation: we recall indeed that the following complex interpolation is true (see~\cite[Theorem 6.4.5]{bl}, whose proof only relies on the dyadic decomposition and may be easily adapted to our situation)
\begin{equation}\label{interpolationresult2}
[  B^{1-\varepsilon}_{p,q};  B^{1+\varepsilon}_{p,q} ]_{\frac12 } =    B^{1}_{p,q} \, .
\end{equation}
The multilinear interpolation result of~\cite[Theorem 4.4.1]{bl}   provides the case~$s=1$ and the other integer cases are obtained similarly. \qed

\medskip
\noindent Note that the above proof actually gives the following result.
\begin{pro}\label{moregeneralresult2}
 Let $\G$ be a Lie group with polynomial volume growth.  
 
 \noindent For every~$s \geq 1$, $1 \leq q \leq \infty$ and~$1<p<\infty$ one has, writing~$1/p = 1/a_i + 1/b_i$ and choosing~$1<a_i,b_i< \infty$,
 $$
 \|fg\|_{  B^s_{p,q}} \leq \|f\|_{L^{a_1}} \|g\|_{  B^s_{b_1,q}} +  \|g\|_{L^{a_2}} \|f\|_{  B^s_{b_2,q}}.
  $$ 
   
 \end{pro}

 \section{Proof of Theorem~\ref{mainnil}}\label{proofthmmainnil}
As in the previous case
 the idea is to argue by induction for the noninteger values of~$s$, and then by interpolation. To
   do so, we need the following result, which is new to our knowledge, even in the context of the Heisenberg group. 
 \begin{pro}\label{bs+1bs} Let~$\G$ be a nilpotent Lie group and let $s >0$ and~$p \in (1,\infty)$ be given. Then $f \in  \dot  B^{s+1}_{p,q}(\G)$ if and only if for all $i=1,...,k$, we have~$X_i f \in   \dot B^{s}_{p,q}(\G). $
 \end{pro}
 \begin{proof}
  On the one hand we need to prove that  for all $i=1,\dots,k$ and~$j \in \N$, 
 $$
 \|\Delta_j X_k f\|_{L^p}  \lesssim 2^j \|\Delta_j   f\|_{L^p}.
 $$
Using Bernstein's inequality and by density of polynomials in the space of continuous functions it is then actually enough to prove that for all integers~$m$,
 \begin{equation}\label{enoughtoprove}
  \| (-\Delta_\G)^\frac m2  (-\Delta_\G)^{-\frac {1}2}X_k    f\|_{L^p}  \lesssim  \| (-\Delta_\G)^\frac m2   f\|_{L^p}.
 \end{equation}
 Indeed if~(\ref{enoughtoprove}) holds, then one also has, multiplying both sides of the inequality by~$2^{jm}$,
 $$
  \| (-2^{2j}\Delta_\G)^\frac m2  (-\Delta_\G)^{-\frac {1}2}X_k    f\|_{L^p}  \lesssim  \| (-2^{2j}\Delta_\G)^\frac m2   f\|_{L^p}
 $$
  so for smooth compactly supported function~$\varphi$, we get by functional calculus
 \begin{equation}\label{alsohas}
 \| \varphi(-2^{2j}\Delta_\G)   (-\Delta_\G)^{-\frac {1}2}X_k    f\|_{L^p}  \lesssim  \| \varphi( -2^{2j}\Delta_\G)    f\|_{L^p}
 \end{equation}
But recalling that~$ \Delta_j = \psi (-2^{2j}\Delta_\G)$ we have
$$
\begin{aligned}
  \| \Delta_j   X_k    f\|_{L^p}  & =    \|  \psi(-2^{2j}\Delta_\G)    X_k    f\|_{L^p}\\
   & =    \| \psi(-2^{2j}\Delta_\G)    (-\Delta_\G)^{ \frac {1}2}  (-\Delta_\G)^{-\frac {1}2}  X_k    f\|_{L^p}
 \, .    \end{aligned}
$$
Then we can write
$$
 \begin{aligned}
 \| \psi(-2^{2j}\Delta_\G)   (-\Delta_\G)^{ \frac {1}2}   (-\Delta_\G)^{-\frac {1}2}  X_k    f\|_{L^p}
       &\lesssim 2^{ j}   \| \psi(-2^{2j}\Delta_\G)    (-\Delta_\G)^{-\frac {1}2}  X_k         f\|_{L^p}  \\
       &\lesssim 2^{ j}   \| \Delta_j   f\|_{L^p}
  \end{aligned}
$$
 due to Bernstein's inequality
  $$
  \|\Delta_j   (-\Delta_\G)^{ \frac {1}2}   f\|_{L^p}
\lesssim 2^{ j}   \| \Delta_j   f\|_{L^p}
  $$
  along with~(\ref{alsohas}). So let us prove~(\ref{enoughtoprove}). 
Actually according to~\cite{lohouevaropoulos} the operator~$ {\mathcal L}_m^k =  (-\Delta_\G)^{\frac {m-1}2} X_k  (-\Delta_\G)^{-\frac m2 }$ is bounded over~$L^p(\G)$  for~$1<p<\infty$.  That property is false if the group is not nilpotent (see for instance~\cite{alexopoulos}) so it
  is here that the assumption that~$\G$ is nilpotent is used. So writing
\begin{eqnarray*}
 (-\Delta_\G)^\frac m2  (-\Delta_\G)^{-\frac {1}2}X_k    f  &=&   (-\Delta_\G)^\frac m2  (-\Delta_\G)^{-\frac {1}2}X_k   (-\Delta_\G)^{-\frac m2 } (-\Delta_\G)^{ \frac m2 }  f \\
 &=&  {\mathcal L}_m^k  (-\Delta_\G)^{ \frac m2 }  f,
\end{eqnarray*}
the result follows.
  
\noindent On the other hand, using again the fact that polynomials are dense in the space of continuous functions, we also need to check that for all~$f$,
 $$
 \|(-\Delta_\G)^\frac {m+1}2  f\|_{L^p}  \lesssim \sup_k \|(-\Delta_\G)^\frac m2 X_k  f\|_{L^p}.
 $$
 To prove that we simply use again the fact that~$ {\mathcal L}_m^k $ is bounded over~$L^p(\G)$  for every index~$1<p<\infty$. Indeed we can write
 $$
 \begin{aligned}
  \|(-\Delta_\G)^\frac {m+1}2  f\|_{L^p} &\leq \sum_k  \|(-\Delta_\G)^\frac {m-1}2 X_k^2  f\|_{L^p} \\
    & =  \sum_k  \|(-\Delta_\G)^\frac {m-1}2 X_k   (-\Delta_\G)^{-\frac m2 } (-\Delta_\G)^{\frac m2 }  X_k f\|_{L^p} \\
     & =  \sum_k \| {\mathcal L}_m^k  (-\Delta_\G)^{\frac m2 }  X_k f\|_{L^p} 
 \end{aligned}
 $$
 whence the result.
Proposition~\ref{bs+1bs} is proved.  \end{proof} 

\noindent Proposition~\ref{bs+1bs} allows to obtain rather easily Theorem~\ref{mainnil} when~$s \in \R^+ \setminus \N$, using also Proposition~\ref{moregeneralresult}.  Let us give the details.

\medskip
\noindent The fact that~$\dot B^s_{\frac ds,1} (\G)$ is embedded in~$L^\infty (\G)$ simply follows from the easy calculations:
$$
\begin{aligned}
\|f\|_{L^\infty (\G)} & \leq \sum_{j \in \Z} \|\Delta_jf\|_{L^\infty (\G)} \\
& \lesssim  \sum_{j \in \Z} 2^{js}  \|\Delta_jf\|_{L^{\frac ds} (\G)}
\end{aligned}
$$
by the Bernstein inequality (Proposition~4.2 of~\cite{furioli}).

\medskip
\noindent
Now let us prove that~$\dot B^s_{\frac ds,1} (\G)$ is an algebra, and then let us prove that
for every~$1 < s$ and every~$1<p <\infty$,  if~$f$ and~$g$ belong to~$ \dot B^s_{p,\frac{s-1}s  } \cap L^{\infty}(\G) $  then~$fg $ belongs to~$  \dot B^s_{p,1} \cap L^{\infty}(\G)$.
We   follow the lines of the inhomogeneous case treated above, but we   need to be careful 
that the norms are now homogeneous.  Let us define~$s = 1+s'$ with~$s' \in (0,1)$.
We write as in the inhomogeneous case, by the Leibniz rule,
$$
 \|X_i (fg)\|_{\dot B^{s'}_{p,q}}  \leq  \|f X_i g\|_{\dot B^{s'}_{p,q}}  +  \|gX_i f\|_{\dot B^{s'}_{p,q}} 
$$
and study more particularly the first term on the right-hand side, which satisfies due to Proposition~\ref{moregeneralresult}, for~$1/a_i + 1/b_i = 1/p$ (and choosing from now on~$1 < a_i,b_i < \infty$),
\begin{equation}\label{firststep}
 \|f X_i g\|_{\dot B^{s'}_{p,q}}  \lesssim   \|f\|_{L^{a_1}}  \|X_i g\|_{\dot B^{s'}_{b_1,q}} + \| f\|_{\dot B^{s'}_{a_2,q}}  \|X_ig\|_{L^{b_2}} .
\end{equation}
On the one hand
$$
 \|X_i g\|_{\dot B^{s'}_{b_1,q}} \lesssim \|g\|_{\dot B^{s}_{b_1,q}},
$$
so it suffices to estimate~$ \| f\|_{\dot B^{s'}_{a_2,q}}  \|X_ig\|_{L^{b_2}}$.

\noindent In the case when~$q=1$ and~$p = d/s$ we choose~$a_2 = d/(s-1) = d/s'$ and~
$b_2 = d$ and use Bernstein's inequality which states that
$$
\|f\|_{\dot B^{s'}_{a_2,1}} \lesssim \|f\|_{\dot B^s_{\frac ds,1} }.
$$
Then according to~\cite{crt} we have
$$
\|X_ig\|_{L^{d}} \lesssim \|(-\Delta_\G)^\frac s2 g\|_{L^{\frac ds}}^\frac 1s \| g\|_{L^{\infty}}^{1-\frac 1s}
$$
and moreover
$$
 \begin{aligned} \|(-\Delta_\G)^\frac s2 g\|_{L^{\frac ds}} 
 &\leq \sum_{j \in \Z} \|\Delta_j (-\Delta_\G)^\frac s2 g\|_{L^{\frac ds}}  \\
  &\lesssim \sum_{j \in \Z} 2^{js} \|\Delta_j  g\|_{L^{\frac ds}} \\
    &\lesssim\|g\|_{\dot B^{s}_{\frac ds,1}}  
\end{aligned}$$
by Bernstein's inequality, so
we infer that
$$
\|X_ig\|_{L^{d}}  \lesssim\|g\|_{\dot B^{s}_{\frac ds,1}}^\frac 1s \| g\|_{L^{\infty}}^{1-\frac 1s}
$$
and the result follows in the case~$1<s<2$. The other noninteger cases  are obtained by induction. To prove the result in the integer case we use a nonlinear interpolation argument as in the inhomogeneous case above: let us detail the case~$s=1$ for instance. We have indeed  (as pointed out in~\cite{bl}, the interpolation results hold in the homogeneous case):
$$
[\dot B^{1-\varepsilon}_{\frac d{1-\varepsilon},1} ,\dot B^{1+\varepsilon}_{\frac d{1+\varepsilon},1} ]_\frac12 = \dot B^1_{d,1}
$$
so the result follows. The other cases are obtained similarly.

\medskip
\noindent In the case when~$f$ and~$g$ belong to~$\dot B^s_{p,\frac{s-1}s}$ then we use as above the fact that  
$$
  \|f\|_{L^{a_1}}  \|X_i g\|_{\dot B^{s'}_{b_1,q}} \lesssim   \|f\|_{L^{a_1}}   \|g\|_{\dot B^{s}_{b_1,q}},
$$
and in particular we can take~$a_1 = \infty$ and~$b_1 = p$, and we choose~$a_2 = ps/(s-1)$ and~$b_2 = ps$. Then H\"older's inequality gives
$$
 \begin{aligned}
2^{js'}\|\Delta_j f\|_{L^{\frac{ps}{s-1}}}& \lesssim  2^{js'}\|\Delta_j f\|_{L^p}^\frac{s-1}s  \|\Delta_j f\|_{L^\infty}^\frac{1}s   \\
& \lesssim \big(  2^{js}\|\Delta_j f\|_{L^p}\big)^\frac{s-1}s \|f\|_{L^\infty}^\frac1s.
\end{aligned}
$$
Since as above
$$
 \begin{aligned}
\|X_ig\|_{L^{ps}} &\lesssim \|(-\Delta_\G)^\frac s2 g\|_{L^p}^\frac 1s \| g\|_{L^{\infty}}^{1-\frac 1s} \\
&\lesssim\|g\|_{\dot B^{s}_{p,1}}^\frac 1s \| g\|_{L^{\infty}}^{1-\frac 1s}
\end{aligned}
$$
the result follows in the case~$1<s$ with~$s$ non integer. Let us detail for instance how to recover the case~$s=2$: we write
for instance
$$
[\dot B^\frac94_{p,\frac59} ,\dot B^\frac95_{p,\frac49} ]_\frac59 = \dot B^2_{p,\frac12}
$$
so the result follows by bilinear interpolation. The other cases are obtained similarly.

\medskip
\noindent Finally let us turn to the last statement of the theorem, namely the fact that if~$1<p_1,p_2<\infty$ with~$1/p = 1/p_1+1/p_2$, if~$1 \leq q \leq \infty$, and if~$f $ belongs to~$\dot B^s_{p_1,q  } \cap L^{p_1}(\G) $ and~$g$ belongs to~$\dot B^s_{p_2,q  } \cap L^{p_2}(\G) $ then~$fg \in \dot B^s_{p,q}\cap L^p (\G)$.
We recall the real interpolation     result~\cite[Theorem 6.4.5]{bl}, which holds also in the homogeneous case as indicated in~\cite{bl}, according to which
\begin{equation}\label{interpolationresult}
[\dot B^{s_1}_{p ,q_1};\dot B^{s_2}_{p,q_2} ]_{\theta,r} = \dot B^{s}_{p,r}   \, , \quad s = \theta s_1 + (1-\theta)s_2 \,  , \quad s_1 \neq s_2
\end{equation}
 so in particular
$$
[\dot B^{0}_{a_2,\infty};\dot B^{s}_{a_2,q} ]_{\frac{s-1}s,1} = \dot B^{s'}_{a_2,1} \hookrightarrow\dot B^{s'}_{a_2,q} \, .
$$
We infer that~$f$ belongs to~$\dot B^{s'}_{a_2,q}$ as soon as~$f $ belongs to~$ L^{a_2 } \cap \dot B^{s}_{a_2,q} \hookrightarrow \dot B^{0}_{a_2,\infty}\cap\dot B^{s}_{a_2,q}$.
Similarly
$$
[\dot B^{0}_{b_2,\infty};\dot B^{s}_{b_2,q} ]_{\frac1s,1} = \dot B^1_{b_2,1}
$$
so using the fact that
$$
\begin{aligned}
\| X_i g\|_{L^{b_2}} &  \lesssim  \| X_i g\|_{\dot B^0_{b_2,1}} \\
&   \lesssim  \|  g\|_{\dot B^1_{b_2,1}}
\end{aligned}
$$
we infer that~$X_i g$ belongs to~$L^{b_2}$ as soon as~$g \in L^{b_2} \cap \dot B^{s}_{b_2,q} \subset\dot B^{0}_{b_2,\infty}\cap\dot B^{s}_{b_2,q} $.
The result follows for~$1 <s<2$, and the theorem is proved by an easy induction, and interpolation as in the inhomogeneous case. \qed

\section{Paradifferential calculus on H-type groups}\label{paradiff}
\setcounter{equation}{0}
In this section, we describe several topics related to   harmonic analysis on H-type groups, which we recall are particular cases of Carnot groups  where it turns out that an explicit Fourier transform is available. 

\subsection{Fourier transforms} \label{Fouriertrans}

In order to construct para-differential and pseudo-differential calculus on H-type groups, one needs to introduce a suitable Fourier transform. This is classically done through the infinite-dimensional unitary irreducible representations on a suitable Hilbert space since H-type groups are non commutative. Two representations are available: the Bargmann representation (see~\cite{kaplanRicci} for instance) and the Schr\"odinger representation (see~\cite{corwin} for instance).

\subsubsection{General definitions}
Let us define generally what a Fourier transform is on non commutative groups. Consider a Hilbert space $\mathcal H_\lambda (\mathbb K^\ell)$ of functions defined on some field $\mathbb K$ (which can be~$\R$ or~$\C$). The irreducible unitary representations
$
\displaystyle \pi_\lambda : \G \to \mathcal H_\lambda (\mathbb K^\ell)
$ are parametrized by $\lambda \in \R^n \backslash \left \{ 0\right \} $ where~$n$ is the dimension of the center of the group.
We have then the following definition. 
\begin{defi}
We define the Fourier transform on $\G$ by the following formula: let $f \in L^1(\G)$. Then the Fourier transform of $f$ is the operator on $\mathcal{H}_\lambda (\mathbb K^\ell)$ parametrized by $\lambda \in \R^n \backslash \left \{0 \right \}$ defined by 
$$\mathcal F (f)(\lambda)=\int_\G f(z,t)\pi_\lambda(z,t)\,dz\,dt.$$ 
\end{defi}
 \noindent  Note that one has $  {\mathcal F}( f \star g )( \lam ) = {\mathcal
F}(f) ( \lam ) \circ {\mathcal F}(g )( \lam ). $ 
   Let $F_{\alpha,\lambda}$, $\alpha \in \mathbb N^\ell$ be a  Hilbert basis of $\mathcal H_\lambda (\mathbb K^\ell)$. 
We recall that an operator $A(\lam)$ of ${\mathcal H}_\lam$ such that
$$ \sum_{\al \in \N^\ell} \left|(A(\lam)F_{\al,  \lam}, F_{\al,
   \lam})_{{\mathcal H}_{\lam}}\right|<+\infty$$
   is said to be of {\it trace-class}. One then sets
$\displaystyle 
    {\rm tr}   \left(A(\lam)\right)=
   \sum_{\al \in \N^\ell} (A(\lam)F_{\al,  \lam}, F_{\al,
   \lam})_{{\mathcal H}_{\lam}}, 
$
and the following inversion theorem holds.
\begin{theo}
\label{inversionth}
  If a function~$f$ satisfies
$\displaystyle
\sum_{\al \in \N^d} \int_{\R^n}
\| {{\mathcal F}(f)(\lam ) F_{\al, \lam} }\|_ {{\mathcal
H}_{\lam}}  |\lam |^{\ell} d\lam < \infty
$
then we have for almost every~$w$,
 $$
 f(w)=
\frac{2^{\ell-1}}{\pi^{\ell + 1}} \int_{\R^n} {\rm tr}
\left( \pi_{\lam}({w^{-1}}){\mathcal F}(f)(\lam )\right)
  |\lam |^{\ell} d\lam.
  $$
   \end{theo}
\noindent Following \cite{yang}, the representation $\pi_\lambda$ on $\G$ 
determines a representation $\pi_\lambda^*$ on its Lie algebra~$\mathcal G$ on the space of $C^\infty$ vectors. The representation $ \pi_\lambda^*$
is defined by 
$\displaystyle\pi^*_\lambda(X)f=\Big ( \frac{d}{dt} \pi_\lambda(\exp(tX))f \Big )|_{t=0}$
for every $X$ in the Lie algebra $\mathcal G$. We can extend $\pi_\lambda^*$ to the 
universal enveloping algebra of left-invariant differential operators on $\G$. 
Let $\mathcal K$ be a left-invariant operator on $\G$, then we have 
$$\mathcal K (\pi_\lambda f,g)=(\pi_\lambda \pi^*_\lambda(\mathcal K)f,g) $$
where $(\cdot , \cdot)$ stands for the $\mathcal H_\lambda$ inner product.

\subsubsection{Bargmann representations on H-type groups}
Given $\lambda \in \R^n \backslash \left \{ 0 \right \}$, consider the Hilbert space (called the Fock space) $\mathcal H_\lambda (\C^\ell) $ of all entire holomorphic functions $F$ on $\C^\ell$ such that 
$\displaystyle \|F\|^2_{\mathcal H_\lambda}=\Big (\frac{2|\lambda|}{\pi} \Big)^\ell\int_{\C^\ell} |F(\xi)|^2 e^{-|\lambda||\xi|^2}d\xi $ is finite.
 The corresponding irreducible unitary representation $\pi_\lambda$ of the group~$\G$ is realized on $\mathcal H_\lambda (\C^\ell)$ by (recall that $t \in \R^n$ and $z,\xi \in \C^\ell$) (see \cite{dooley})
$$(\pi_{\lambda}(z,t)F)(\xi)=F(\xi-z)e^{i \langle \lambda,t \rangle - |\lambda |(|z|^2+\langle z,\xi \rangle)}.$$
 It is a well-known fact that the Fock space admits an orthonormal basis given by the monomials 
$\displaystyle
F_{\alpha,\lambda}(\xi)= \frac{(\sqrt{2 |\lambda|}\,\xi)}{\sqrt{\alpha\, !}}^\alpha,\,\,\,\alpha \in \mathbb N^\ell .
$
 A very important property for us is the following classical diagonalization result (see~\cite{kaplanRicci}, or~\cite{bfg} and the references therein). 
\begin{pro}
Let $\mathcal F_B$ be the   Fourier transform associated to the Bargmann representation~$  \pi$.The following diagonalization property holds: for every $f \in \mathcal S(\G)$,
$$
\mathcal F_B(\Delta_\G f) (\lambda)F_{\alpha,\lambda}=-4|\lambda|(2|\alpha|+\ell)\mathcal F_B(f) (\lambda)F_{\alpha,\lambda} \, .
$$
\end{pro}

\noindent This allows to define the following formula, for every $\rho \in \mathbb R$:
$$\mathcal F_B((-\Delta_\G)^{\rho} f) (\lambda)F_{\alpha,\lambda}=(4|\lambda|(2|\alpha|+\ell))^\rho\mathcal F_B(f) (\lambda)F_{\alpha,\lambda}.$$

\subsubsection{The $L^2$ representation on H-type groups}
Another useful representation is the so-called Schr\"odinger, or $L^2$ representation. In this case, the  unitary irreducible representations are given on $L^2(\R^\ell)$ by, for $\lambda \in \R^n $ (and writing~$z=(x,y)$):
$\displaystyle(\tilde \pi_\lambda(z,t)F)(\xi)=e^{i\langle \lambda, t \rangle +|\lambda | i( \sum_{j=1}^\ell x_j \xi_j +\frac12 x_j y_j)}F(\xi+y). $
The intertwining operator between the Bargmann and the $L^2$ representations is
the Hermite-Weber transform~$ \displaystyle K_\lambda:{\mathcal
H}_\lambda\rightarrow  L^2(\R^\ell)
$
 given by
$$
  (K_\lambda
\phi)(\xi)=C_{\ell}|\lambda|^{\ell/4} {\rm
e}^{|\lambda|\frac{|\xi|^2}{2}}
\phi\left(-\frac{1}{2|\lambda|}\frac{\partial}{\partial\xi}\right){\rm
e}^{-|\lambda|\,|\xi|^2}, 
$$
which is unitary and satisfies~$\displaystyle K_\lambda
\pi_\lambda(z,t)=\tilde \pi_\lambda(z,t) K_\lambda .$ 
Following \cite{yang} and the previous description, we have
$\displaystyle\tilde \pi^*_\lambda (X_j)=i|\lambda |\xi_j $ and~$\displaystyle \tilde \pi^*_\lambda (Y_j)=\frac{\partial}{\partial \xi_j}$
for $j=1,...,\ell$, and similarly for~$k=1,...,n$, $\displaystyle\tilde  \pi^*_\lambda(\partial_{t_k} )=i \lambda_k$. Therefore, we have 
$$
\tilde \pi^*_\lambda (-\Delta_\G)=-\sum_{j=1}^n \frac{\partial^2}{\partial \xi_j^2}+|\lambda|^2 |\xi|^2. 
$$
Notice  that this is a Hermite operator and the eigenfunctions of $\pi^*_\lambda (-\Delta_\G)$ are 
$\displaystyle\Phi_\alpha^\lambda(\xi)=|\lambda |^{n/4}\Phi_\alpha (\sqrt{|\lambda |} \xi),$ $ \alpha=(\alpha_1,...,\alpha_\ell)$
where $\Phi_\alpha(\xi)$ is the product $\psi_{\alpha_1}(\xi_1)...\psi_{\alpha_\ell}(\xi_\ell)$ and 
$\psi_{\alpha_j}(\xi_j)$ is the eigenfunction of $-\frac{\partial^2}{\partial \xi_j^2}+\xi_j^2$
with eigenvalue~$2\alpha_j+1$. This leads to the following formula, where $|\alpha|=\alpha_1+...+\alpha_\ell:$
$$\tilde \pi^*_\lambda (-\Delta_\G)\Phi_\alpha^\lambda=(2|\alpha|+\ell)|\lambda| \Phi_\alpha^\lambda.$$
As a consequence, one has the following lemma. 
\begin{lemma}
Let $\mathcal F_S$ be the   Fourier transform associated to the Schr\"odin\-ger representation~$\tilde \pi$.The following diagonalization property holds: for every~$f $ in~$ \mathcal S(\G)$, 
$$
\displaystyle\mathcal F_S((-\Delta_\G)f)\Phi_\alpha^\lambda=(2|\alpha|+\ell) |\lambda| \mathcal F_S(  f) \Phi_\alpha^\lambda \, .
$$
\end{lemma}
\begin{proof}
We have by definition 
$$
\mathcal F_S((-\Delta_\G)f)\Phi_\alpha^\lambda = \int_\G (-\Delta_\G)f(z,t)\tilde \pi_\lambda(z,t)\Phi_\alpha^\lambda = \int_\G f(z,t)(-\Delta_\G)\tilde \pi_\lambda(z,t)\Phi_\alpha^\lambda.
$$
Using the definition of the dual representation, we have 
$$\mathcal F_S((-\Delta_\G)f)\Phi_\alpha^\lambda=\int_\G f(z,t)\pi_\lambda(z,t)\pi^*_\lambda(-\Delta_\G)\Phi_\alpha^\lambda$$
and using the properties of the Hermite operator, this gives   the   result. 
\end{proof}

\subsection{A localization lemma }\label{loclemma}
As in \cite{bg}, one can prove a localization lemma (also called Bernstein Lemma), which we     state here in the context of the Bargmann representation.  The
proof is omitted as it is identical to the Heisenberg situation treated in~\cite{bg}. Note that using  Proposition~\ref{bs+1bs}, the  last statement of the lemma could be extended to iterated vector fields~$X^I$. We denote $\mathcal C_0$ the ring $\left \{ \tau \in \mathbb R\,\,1/2 \leq |\tau | \leq 4 \right \}$ and by $\mathcal B_0$ the ball  $\left \{ \tau \in \mathbb R\,\, |\tau| \leq 2 \right \}$. 
\begin{lem}  \label{lem:lech}
Let $p$ and $q$ be two elements of~$[1,\infty]$, with~$p \leq q$, and let~$u \in  {\mathcal S}(\G)$ satisfy  for all~$\alpha \in \N^\ell$,
$ {\mathcal F}_B (u)(\la) F_{\alpha, \lam}   = {\bf 1}_{ \lam \in (2|\alpha|+\ell)^{-1}2^{2j}{\mathcal B}_0}  {\mathcal F}_B (u)(\la) F_{\alpha, \lam}  .
 $ Then  we have 
$$\forall k \in \N, \quad    \sup_{|\beta | = k} \|\X^\beta  u\|_{L^q(\G)} \leq
C_k 2^{Nj(\frac{1}{p}-\frac{1}{q})+kj} \| u\|_{L^p(\G)}.
$$
On the other hand, if 
${\mathcal F} _B(u)(\la)F_{\alpha,\lambda} = 
 {\bf 1}_{ \lam \in (2|\alpha|+\ell)^{-1}2^{2j}{\mathcal C}_0}   ,
$
then for all~$ \rh \in \R$,
$$
    C_\rh^{-1} 2^{-j \rh}  \| (-\Delta_{\G})^{\frac{\rh}{2}} u\|_{L^p(\G)} \leq \| u\|_{L^p(\G)} 
\leq C_\rh 2^{-j \rh}  \| (-\Delta_{\G})^{\frac{\rh}{2}} u\|_{L^p(\G)}.
$$
\end{lem}
\subsection{Paraproduct on H-type groups}\label{paraproduct}
In order to develop a paraproduct on H-type groups, one has to prove that the product of two functions is localized in frequencies whenever the functions are localized. This is the object of the next lemma, whose proof is the same as that of Proposition 4.2 of~\cite{bg}.  

 \begin{lem}\label{loc}There is a constant $M_1\in\N$ such that the following holds. Consider~$f$ and $g$ two functions of~${\mathcal S}(\G)$ such that
$$
{\mathcal F} _B(f)(\lam)F_{\alpha,\lambda} =  \displaystyle {\bf 1}_{\lam \in(2|\alpha|+\ell)2^{2m}{\mathcal C}_0}(\lam){\mathcal F} _B(f)(\lam)F_{\alpha,\lambda} \quad \mbox{and} 
$$
$$
{\mathcal F} _B(g)(\lam) F_{\alpha,\lambda} =  \displaystyle{\bf 1}_{\lam \in(2|\alpha|+\ell)2^{2m'}{\mathcal C}_0}(\lam){\mathcal F} _B(g)(\lam)F_{\alpha,\lambda} 
$$
  for   $m$ and $m'$  integers. If
   $m'-m>M_1$, then there exists a ring~$\widetilde{\mathcal C}$ such that
  $${\mathcal F }_B(fg)(\lam)F_{\alpha,\lambda} ={\bf 1}_{\lam \in(2|\alpha|+\ell)2^{2m'} \widetilde{\mathcal C}}(\lam){\mathcal F} _B(fg)(\lam)F_{\alpha,\lambda} .$$
  On the other hand, if $|m'-m|\leq M_1$, then there exists a ball $\tilde{\mathcal B}$ such that
  $${\mathcal F} _B(fg)(\lam)F_{\alpha,\lambda} ={\bf 1}_{\lam \in(2|\alpha|+\ell)2^{2m'} \widetilde{\mathcal B}}(\lam){\mathcal F} _B(fg)(\lam)F_{\alpha,\lambda} .$$ 
  \end{lem}
\begin{defi}We shall call paraproduct of~$v$ by~$u$ and shall denote by~$T_u v$ the   bilinear operator~$
\displaystyle T_u v =  \sum_{j} S_{j-1}u \,\D_jv.
$ We shall call remainder of~$u$ and~$v$ and shall denote by~$R(u,v)$ the   bilinear operator~$
 \displaystyle R(u,v) =  \sum_{|j-j'|\leq 1} \D_ju \,\D_{j'}v.
 $
 \end{defi}
 \begin{remark} 
It is clear that formally
$
uv = T_u v + T_v u + R(u,v). 
$
\end{remark}
\noindent One of the classical consequences of Lemma~\ref{loc} is the following result, which is obtained using the previous decomposition as well as localization properties of the paraproduct and remainder terms.
\begin{cor}\label{productlaws}
 Let $\rh >0$ and $(p,r) \in [1,+\infty]^2$ be three real numbers. Then
 $$
\|fg\|_{B^{\rh }_{p ,r }(\G)} \leq C
(\|f\|_{L^\infty}\|g\|_{B^{\rh }_{p ,r }(\G)}+\|g\|_{L^\infty}\|f\|_{B^{\rh }_{p ,r }(\G)}).
$$
If $\rh_1 + \rh_2 > 0$ and if $p_1$ is such that $\rh_1 < Q/p_1$, then for all $(p_2,r_2) \in [1,+\infty]^2$  writing $\rh = \rh_1 + \rh_2 - Q/p_1$,
$$
\|fg\|_{B^{\rh}_{p_2,r_2}(\G)} \leq C (\|f\|_{B^{\rh_1}_{p_1,\infty}}
\|g\|_{B^{\rh_2}_{p_2,r_2}}+\|g\|_{B^{\rh_1}_{p_1,\infty}}
\|f\|_{B^{\rh_2}_{p_2,r_2}}).
$$
\noindent Moreover, if ~$\rh_1 + \rh_2 \geq 0$,  $\rh_1 < Q/p_1$ and~$\displaystyle \frac{1}{r_1} + \frac{1}{r_2} = 1$, then $$
\|fg\|_{B^{\rh}_{p,\infty}(\G)} \leq C (\|f\|_{B^{\rh_1}_{p_1,r_1}} \|
g\|_{B^{\rh_2}_{p_2,r_2}}+\|g\|_{B^{\rh_1}_{p_1,r_1}}
\|f\|_{B^{\rh_2}_{p_2,r_2}}).
$$
 Finally if  $\rh_1 + \rh_2 > 0$, $\rh_j < Q/p_j$ and $p \geq
 \max(p_1,p_2)$, then for all~$(r_1,r_2)$,
$$
\|fg\|_{B^{\rh_{12}}_{p,r}(\G)} \leq C \|f\|_{B^{\rh_1}_{p_1,r_1}}
\|g\|_{B^{\rh_2}_{p_2,r_2}},
$$
with $\displaystyle \rh_{12} = \rh_1+\rh_2 -
Q(\frac{1}{p_1}+\frac{1}{p_2}-\frac{1}{p})$ and $r = \max(r_1,r_2)$, and if~$\rh_1 +
\rh_2 \geq 0$, with~$\rh_j < Q/p_j$ and with~$\displaystyle \frac{1}{r_1} +
\frac{1}{r_2} = 1$, then for all~$p
\geq \max(p_1,p_2)$,
$$
\|fg\|_{B^{\rh_{12}}_{p,\infty}(\G)} \leq C \|f\|_{B^{\rh_1}_{p_1,r_1}}
\|g\|_{B^{\rh_2}_{p_2,r_2}}.
$$
\end{cor}
\noindent The same results hold in the case of homogeneous Besov spaces.  Once the paraproduct algorithm is in place, one can obtain (refined)  Sobolev and Hardy inequalities (see~\cite{cheminxu} and~\cite{bg}  for Sobolev embeddings in the euclidean case and for the Heisenberg group,   and~\cite{hardyBCG} for the Hardy inequalities --  see also~\cite{chamorropaper} for recent extensions).
  One can also construct, in the context of H-type groups, an algebra of pseudo-differential operators exactly as on the Heisenberg group. We refer to~\cite{bfg} for details.

\end{document}